\documentclass[11pt]{article}
\usepackage{amssymb,amsmath,amsthm}
\usepackage{epsfig}
\usepackage{breqn}
\usepackage{rotating}
\usepackage{amsfonts}
\usepackage{setspace}
\usepackage{fullpage}
 \usepackage{bbold} 
\usepackage{amsmath}
\usepackage{comment}
\usepackage{pgf,tikz}
\usepackage{mathtools}  
\usepackage{hyperref}

\DeclareMathOperator{\ex}{ex}
\DeclareMathOperator{\Aut}{Aut}
\let\epsilon\varepsilon

\usepackage{authblk}

\usepackage{setspace}
\usepackage{bbm}
\newtheorem{theorem}{Theorem}

\newtheorem{conjecture}{Conjecture}

\newtheorem*{observation*}{Observation}

\newcommand{\floor}[1]{\left\lfloor{#1}\right\rfloor}

\newcommand{\abs}[1]{\left\lvert{#1}\right\rvert}

\title{Subgraph densities in $K_r$-free graphs}

\author[1]{Andrzej Grzesik}
\author[2]{Ervin Győri}
\author[2,3]{Nika Salia}
\author[2]{Casey Tompkins}

\affil[1]{Jagiellonian University.}
\affil[2]{Alfr\'ed R\'enyi Institute of Mathematics, Hungarian Academy of Sciences. }
\affil[3]{Asian University for Women. }

\begin{document}

\maketitle

\begin{abstract}
    In this paper we disprove a conjecture of Lidický and Murphy about the number of copies of a given graph in a $K_r$-free graph and give an alternative general conjecture.  We also prove an asymptotically tight bound on the number of copies of any bipartite graph of radius at most~$2$ in a triangle-free graph.
\end{abstract}

\section{Introduction}

For graphs $H$ and $F$ the generalized Tur\'an number $\ex(n,H,F)$ is defined to be the maximum number of (not necessarily induced) copies of $H$ in an $n$-vertex graph $G$ which does not contain $F$ as a subgraph. 
Estimating $\ex(n,H,F)$ for various pairs $H$ and $F$ has been a central topic of research in extremal combinatorics. The case when $H$ and $F$ are both cliques was settled early on by Zykov~\cite{zykov1949some} and independently  by Erd\H{o}s~\cite{erdos1962number}. 
The problem of maximizing $5$-cycles in a triangle-free graph was a long-standing open problem. The problem was finally settled by Grzesik~\cite{grzesik2012maximum} and independently by Hatami, Hladk\'y, Kr\'al, Norine and Razborov~\cite{hatami2013number}. 
In the case when the forbidden graph $F$ is a triangle and $H$ is any bipartite graph containing a  matching on all but at most one of its vertices,  $\ex(n,H,F)$ was determined exactly by Gy\H{o}ri, Pach and Simonovits~\cite{gyori1991maximal} in 1991. 
More recently there has been extensive work on the topic following the work of Alon and Shikhelman~\cite{alon2016many}, who introduced the extremal function $\ex(n,H,F)$ for general pairs $H$ and $F$.  

We now introduce some further notation which we will require in the statements and proofs of our main results.  
For a graph $G$, the vertex set of $G$ is denoted by $V(G)$ and the edge set of $G$ is denoted by $E(G)$.  
We also write $v(G) = |V(G)|$ and $e(G) = |E(G)|$. 
We denote the path, cycle and complete graph on $r$ vertices by $P_r$, $C_r$ and $K_r$, respectively.
The complete multipartite graph with $r \geq 2$ parts of sizes $n_1,n_2,\dots,n_r$ is denoted by $K_{n_1,n_2,\dots,n_r}$.
In the case when each $n_i$ differs by at most one from the others the $n$-vertex graph is referred to as the Tur\'an graph and is denoted by $T_r(n)$.
For a graph $H$, the $k^{\text{th}}$ power of $H$, denoted $H^k$, is defined to be the graph with vertex set $V(H)$ with an edge between every vertex of distance at most $k$ in $H$. 
For graphs $G$ and $H$, the number of labelled copies of $H$ in $G$ is denoted by $H^*(G)$ and the number of unlabelled copies of $H$ in $G$ is denoted by $H(G)$. 
In particular we we have that $H^*(G)/H(G) = |\Aut(H)|$ where $\Aut(H)$ is the set of automorphisms of $H$. 

Recently Lidický and Murphy proposed the following natural conjecture.

\begin{conjecture}[Lidický, Murphy~\cite{lidicky2021maximizing}]\label{conj:Lid_Mer}
Let $H$ be a graph and let $r$ be an integer such that $r>\chi(H)$. 
Then there exist integers $n_1,n_2,\dots,n_{r-1}$ such that $n_1+n_2+\dots+n_{r-1}=n$ and we have 
\[
\ex(n,H,K_r)=H(K_{n_1,n_2,\dots,n_{r-1}}).
\]
\end{conjecture} 

Using the graph removal lemma one can easily show that for any graphs $H$ and $F$ with $\chi(F) = r$ we have $\ex(n,H,F) \leq \ex(n,H,K_r) + o(n^{v(H)})$~(see~\cite{gerbner2019counting}). 
Therefore, the above conjecture asymptotically determines $\ex(n,H,F)$ in the case $\chi(F)> \chi(H)$, which shows its importance. 
Unfortunately, the conjecture is not true in general. Indeed a counterexample when $r=3$ already appears in~\cite{gyori1991maximal}. Here we give a counterexample for arbitrary~$r$. 

\begin{theorem}
For every $r \geq 3$ there is a counterexample to Conjecture~\ref{conj:Lid_Mer}.
\end{theorem}

\begin{proof}
First we fix some constants later used for constructing a counterexample. 
Let $\epsilon$ be a positive real number such that $\epsilon <\frac{1}{4r}$. 
Take a positive integer $a$ for which 
\[
2\epsilon ^{2r-2}(1-(2r-2)\epsilon)^{2a}>\frac{1}{2^{2a}}.
\]

\begin{figure}[ht]
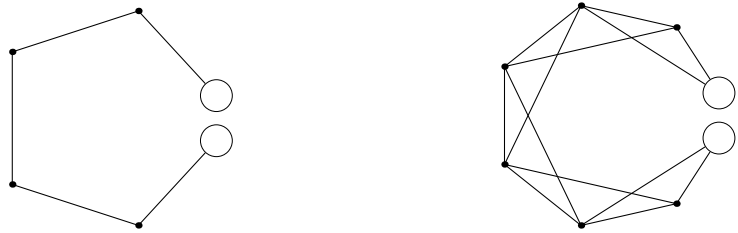

{\hfill\includegraphics{figure-2.mps}
\hfill\includegraphics{figure-4.mps}\hfill}
\caption{Graph $H$ for $r=3$ and $r=4$.}\label{fig:graphH}
\end{figure}

Let $H$ be the graph, depicted in Figure~\ref{fig:graphH}, obtained from $P^{r-2}_{2r}$ by replacing each of the two vertices of degree~$r-2$ with independent sets of size $a$ each with the same neighborhood as the original vertex.
We refer to these $a$ vertices as copies of the terminal vertex. 
Note that there is a unique $(r-1)$-coloring of $H$, and the copies of different terminal vertices are in different color classes.
For integers $n, n_1,n_2,\dots, n_{r-1}$ such that $n=n_1+n_2+\dots+n_{r-1}$, we have
\[
H(K_{n_1,n_2,\dots,n_{r-1}})=\frac{1}{|\Aut(H)|}\cdot H^*(K_{n_1,n_2,\dots,n_{r-1}})\leq n^{2r-2}\left(\frac{n}{2}\right)^{2a}.
\]

\begin{figure}[ht]
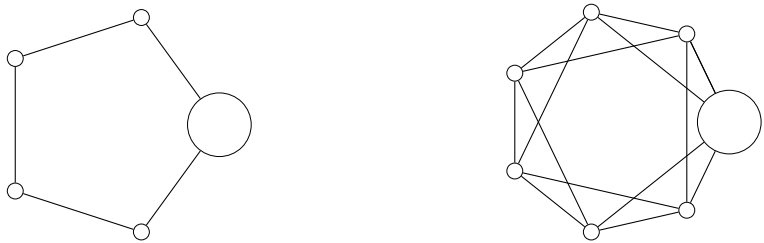

{\hfill\includegraphics{figure-3.mps}
\hfill\includegraphics{figure-5.mps}\hfill}
\caption{Graph $G$ for $r=3$ and $r=4$.}\label{fig:graphG}
\end{figure}

Let $G$ be a graph, depicted in Figure~\ref{fig:graphG}, obtained from blowing up $C_{2r-1}^{r-2}$ in the following way. 
We replace each vertex with a disjoint independent set of size $\floor{\epsilon n}$ except for one vertex which we replace by an independent set $A$ of size $n-(2r-2)\floor{\epsilon n}$. 
Note that $G$ is an $n$-vertex graph and the number of labelled copies of $H$ in $G$ is at least 
\[
2(\floor{\epsilon n})^{2r-2}(n-(2r-2)\floor{\epsilon n})^{2a}+o(n^{2r+2a-2}).
\]
Recall by the choice of $a$ we have 
\[
2(\floor{\epsilon n})^{2r-2}(n-(2r-2)\floor{\epsilon n})^{2a}+o(n^{2r+2a-2}) >n^{2r-2}\left(\frac{n}{2}\right)^{2a}
\]
for large enough $n$.  
Therefore for sufficiently large $n$ the number of labelled copies as well as unlabelled copies of $H$ in $G$ is greater than the number in any $n$ vertex $(r-1)$-partite graph. 
\end{proof}

It may be natural to expect that for $r=3$ the blow-up of a pentagon is the only obstacle, in particular that for every bipartite graph $H$, $\ex(n,H,K_3)$ is asymptotically achieved by either a blow-up of an edge (that is, a complete bipartite graph) or a blow-up of a cycle of length five. 
Surprisingly this is not the case.
Here we give an intuitive sketch of the argument. 
Let $H$ be the first graph depicted in Figure~\ref{fig:C5example} defined in the following way. 
We take a path on $10$ vertices $v_1,v_2,\dots,v_{10}$, and let $A_2$ and $A_9$ be big sets of independent vertices attached to $v_2$ and $v_9$, accordingly. 
Let $B_1$, $B_4$, $B_7$ and $B_{10}$ be huge sets of independent vertices attached to the vertices $v_1$, $v_4$, $v_7$ and $v_{10}$, accordingly. 
If one wants to maximize the number of copies of $H$ in a complete bipartite graph, then the huge independent sets $B_1$, $B_7$ will be mapped into one color class and the huge independent sets $B_4$ and $B_{10}$ will be mapped into the other color class.
If one wants to maximize the number of copies of $H$ in a blow-up of a pentagon, then the huge sets $B_1$, $B_4$, $B_7$ and $B_{10}$ can be mapped into the same part of the blow-up of the pentagon, but the remaining two big sets $A_2$ and $A_9$ need to go to distinct parts, which are also different from the parts into which the huge independent sets are mapped. 
On the other hand, when one counts the number of copies of $H$ in the  graph depicted on the right in Figure~\ref{fig:C5example}, then all sets $B_1$, $B_4$, $B_7$ and $B_{10}$ can be mapped into one huge part and both of the sets $A_2$ and $A_9$ can be mapped into one big part. Therefore, after fixing the sizes of independent sets to appropriate values, the maximum number of copies of $H$ in a triangle-free graph will be achieved neither in a complete bipartite graph nor in a blow-up of a pentagon.

\begin{figure}[ht]
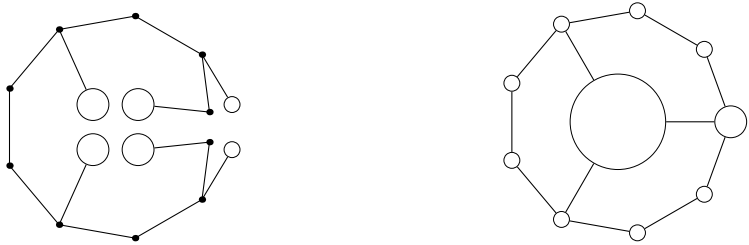

{\hfill\includegraphics{figure-0.mps}
\hfill\includegraphics{figure-1.mps}\hfill}
\caption{The graph $H$ is depicted on the left, and the structure of a graph with more copies of $H$ than a blow-up of an edge or $C_5$ is depicted on the right.}\label{fig:C5example}
\end{figure}

The main idea behind the counterexample we presented to Conjecture~\ref{conj:Lid_Mer} is to have a graph with many vertices that cannot have the same color in any two-coloring, but can be in the same part in a blow-up of a non-bipartite graph. 
One can avoid having such vertices by bounding the diameter of a graph, therefore it is natural to consider the following problem instead of Conjecture~\ref{conj:Lid_Mer}.
\begin{conjecture}\label{Conj:An_G_S_T}
Let $G$ be a graph with diameter at most $2r-2$ with $\chi(G)<r$, then $\ex(n,G,K_{r})$ is asymptotically achieved by a blow-up of $K_{r-1}$.
\end{conjecture}

A first step towards Conjecture~\ref{Conj:An_G_S_T} for $r=3$ is to prove it for all bipartite graphs of radius~$2$. 
Each such graph can be viewed as a star with additional adjacent vertices. Here we prove a bit more general result, i.e., for bipartite graphs consisting of some complete bipartite graph and additional adjacent vertices.

\begin{theorem}
Let $H$ be a bipartite graph containing a subgraph $K$ isomorphic to $K_{s,t}$. Assume the distance of each vertex $v \in V(H)$ to $V(K)$ is at most one. Then the maximum number of copies of $H$ in a triangle-free $n$-vertex graph is obtained asymptotically by a complete bipartite graph.
\end{theorem}

\begin{proof}
We start proof with a simple observation. 
Let us assume that the maximum number of copies of a connected graph $H'$ in a triangle-free $n$-vertex graph is obtained by a blow-up of an edge. 
Then for every bipartite graph $H$ such that $H'\subseteq H$ we have that the maximum number of copies of $H$ in a triangle-free $n$-vertex graph is obtained by a blow-up of an edge.
Therefore we may assume that  $H$ consists of a complete bipartite graph $K_{s,t}$ with color classes $S$ and $T$ and some pendant edges.
The number of pendant edges attached to the vertices of $S$ are denoted $a_1,a_2,\dots,a_s$ and the number of pendant edges attached to vertices of $T$ are denoted $b_1,b_2,\dots,b_t$. 

For a graph $H$, we estimate the number of labeled copies of $H$ in a graph $G$.
First we fix a set of size $s$ in $G$ say $\{x_1,x_2,\dots,x_s\}$, onto which we will map the color class $S$ of $H$. 
Let us denote the common neighborhood of $\{x_1,x_2,\dots,x_s\}$ in $G$ by $X=\bigcap_{i=1}^{s} N_{G}(x_i)$. 
In the estimates below the vertices $x_1,x_2,\dots,x_s$ are variables therefore the common neighborhood is another variable. 
After the set $\{x_1,x_2,\dots,x_s\}$ is chosen we choose a permutation $\sigma\in S_s$ to map vertices of $\{x_1,x_2,\dots,x_s\}$ to the vertices of $S$. 
Next we choose vertices $y_1,y_2,\dots,y_t\in X$ as representatives of $T$.
Finally we choose the endpoints of the pendant edges. We have
\begin{equation}\label{Equation:H(G)bound}
H^*(G)=\sum_{\{x_1,\dots,x_s\} \subset V(G)} \left(\sum_{\sigma \in S_s} \prod_{i=1}^{s} d(x_{\sigma(i)})^{a_i} \right)
        \left(\sum_{y_1,\dots,y_t\in X}\prod_{j=1}^{t} d(y_j)^{b_j}\right)+o(n^{v(H)}).
\end{equation}
We use  Muirhead's inequality~\cite[Theorem 45]{hardy1952inequalities} to estimate both terms of the product above. 
For the degrees of $x_1,x_2,\dots,x_s$ since the sequence $(a_1,a_2,\dots,a_{s})$ is majorized by the sequence $(\sum_{i=1}^{s}a_i,0,\dots,0)$ we have
\begin{equation}\label{Equation:Muir1}
\sum_{\sigma \in S_s} \prod_{i=1}^{s} d(x_{\sigma(i)})^{a_i}\leq (s-1)!\sum_{x\in \{x_1,\dots,x_s\} }  d(x)^{\sum_{i=1}^{s}a_i}.
\end{equation}
Moreover for the degrees of all vertices of $X$ the sequence $(b_1,b_2,\dots,b_{t},0,0\dots,0)$ is majorized by the sequence $(\sum_{i=j}^{t}b_j,0,\dots,0)$ we have 
\begin{equation}\label{Equation:Muir2}
\sum_{y_1,\dots,y_t\in X}\prod_{j=1}^{t} d(y_j)^{b_j}
        \leq \frac{(\abs{X}-1)!}{(\abs{X}-t)!} \sum_{y\in X}d(y)^{\sum_{j=1}^{t} b_j}. 
\end{equation}
Note that we have $\frac{(\abs{X}-1)!}{(\abs{X}-t)!}\leq \abs{X}^{t-1}\leq d(x)^{t-1}$ for all $x$ in $\{x_1,x_2,\dots,x_s\}$. 
Putting together the bounds~\eqref{Equation:H(G)bound},~\eqref{Equation:Muir1} and~\eqref{Equation:Muir2} we obtain
\begin{equation}\label{equation:H(G)bound}
    H^*(G)\leq  \sum_{x\in\{x_1,\dots,x_s\} \subset V(G)} (s-1)!  d(x)^{t-1+\sum_{i=1}^{s}a_i}    \sum_{y\in X}  d(y)^{\sum_{j=1}^{t} b_j} +o(n^{v(H)}) = F^*(G)+o(n^{v(H)}),
\end{equation}
where $F$ is a double-star with central vertices $v$ and $u$ joined by an edge, $\sum_{i=1}^s a_i+t-1$ pendant edges attached to $v$ and $\sum_{i=1}^t b_i+s-1$ pendant edges attached to $u$. 
Here we explain the last equality. 
Let us fix a set $S'$ in $V(F)$ containing the vertex $v$ and $s-1$ leaves adjacent with $u$. 
In order to find a copy of $F$ in $G$ first we choose vertices  $x_1,x_2,\dots,x_s$ of $G$, then we map vertices from $S'$ to it and choose representatives of all vertices adjacent to $v$ in $F$ except $u$. 
Then we fix a vertex $y$ representing $u$, and finally we choose the remaining leaves adjacent to it. 
   
    For a given $n$ and $F$, Gy\H ori, Wang and Woolfson~\cite{gyHori2021extremal} proved that there exists $n'$  such that for all triangle-free graphs $G$ on $n$ vertices we have $F(G)\leq F(K_{n',n-n'})+o(n^{v(F)})$.  
   Therefore we have $F^*(G)\leq F^*(K_{n',n-n'})+o(n^{v(F)})$. 
   Hence the maximum number of labeled copies of $H$ in $G$ is also asymptotically attained when $G=K_{n',n-n'}$ 
   \[
   H(G)\leq H(K_{n',n-n'})+o(n^{v(H)}). \qedhere
   \]
   \end{proof}

\section*{Acknowledgments}

We would like to thank Daniel Gerbner for some useful remarks on the manuscript. The research of Grzesik was supported by the National Science Centre grant 2021/42/E/ST1/00193.
The research of Gy\H{o}ri and Salia was supported by the National Research, Development and Innovation Office NKFIH, grants  K132696 and SNN-135643.  
The research of Tompkins was supported by NKFIH grant K135800. 
\bibliography{References.bib}

\end{document}